\documentclass[11pt,a4paper]{amsart}

\usepackage{amsmath}
\usepackage{amsthm}
\usepackage{mathtools}
\usepackage{amssymb}
\usepackage{hyperref}

\usepackage{tikz}
\usepackage{tikz-cd}
\usetikzlibrary{shapes.geometric,shapes.symbols}
\tikzset{
  arrowlabel/.style={auto,inner sep=1mm,font=\footnotesize},
  confluencediagram/.style={x=15mm,y=15mm,line width=.6pt},
}

\theoremstyle{definition}
\newtheorem{definition}{Definition}[section]

\theoremstyle{plain}
\newtheorem{proposition}[definition]{Proposition}
\newtheorem{lemma}[definition]{Lemma}
\newtheorem{theorem}[definition]{Theorem}
\newtheorem{corollary}[definition]{Corollary}

\numberwithin{equation}{section}

\def\fullref#1#2{%
  \ifdefined\hyperref%
    {\hyperref[#2]{#1\space\penalty 200\relax\ref*{#2}}}%
  \else%
    {#1\space\penalty 200\relax\ref{#2}}%
  \fi%
}

\newcommand{\defterm}[1]{\textit{#1}}

\newcommand{\ch}{ch.~}

\newcommand{\nset}{\mathbb{N}}

\DeclareMathOperator{\im}{im}

\newcommand{\pres}[2]{\left\langle #1\:|\:#2 \right\rangle}

\newcommand{\drel}[1]{\mathcal{#1}}

\newcommand{\emptyword}{\varepsilon}

\newcommand{\tm}[1]{\mathfrak{#1}}

\newcommand{\lenlex}{\mathrm{lenlex}}

\newcommand{\imreduces}{\rightarrow}
\newcommand{\reduces}{\rightarrow^*}

\newcommand{\thue}{\leftrightarrow^*}

\newcommand{\bst}{\mathcal{BST}}
\newcommand{\lrp}{{LRP}}
\newcommand{\tree}[1]{\mathcal{#1}}
\newcommand{\tmcomponestep}{\vdash}
\newcommand{\tmcomp}{\vdash^*}

\begin{document}

\title[Deciding conjugacy]{Deciding conjugacy in sylvester monoids and other homogeneous monoids}

\author[A.J. Cain]{Alan J. Cain}
\address{%
Centro de Matem\'{a}tica e Aplica\c{c}\~{o}es (CMA)\\
Faculdade de Ci\^{e}ncias e Tecnologia\\
Universidade Nova de Lisboa\\
2829--516 Caparica\\
Portugal
}
\address{%
Departamento de Matem\'{a}tica\\
Faculdade de Ci\^{e}ncias e Tecnologia\\
Universidade Nova de Lisboa\\
2829--516 Caparica\\
Portugal
}
\email{%
a.cain@fct.unl.pt
}

\thanks{During the research that led to this paper, the first author was initially supported by the European Regional
Development Fund through the programme {\sc COMPETE} and by the Portuguese Government through the {\sc FCT}
(Funda\c{c}\~{a}o para a Ci\^{e}ncia e a Tecnologia) under the project {\sc PEst-C}/{\sc MAT}/{\sc UI}0144/2011 and
through an {\sc FCT} Ci\^{e}ncia 2008 fellowship, and later supported by an {\sc FCT} Investigador advanced fellowship
({\sc IF}/01622/2013/{\sc CP}1161/{\sc CT}0001).}

\author[A. Malheiro]{Ant\'{o}nio Malheiro}
\address{%
Centro de Matem\'{a}tica e Aplica\c{c}\~{o}es (CMA)\\
Faculdade de Ci\^{e}ncias e Tecnologia\\
Universidade Nova de Lisboa\\
2829--516 Caparica\\
Portugal
}
\address{%
Departamento de Matem\'{a}tica\\
Faculdade de Ci\^{e}ncias e Tecnologia\\
Universidade Nova de Lisboa\\
2829--516 Caparica\\
Portugal
}
\address{%
Centro de \'{A}lgebra da Universidade de Lisboa\\
Av. Prof. Gama Pinto 2\\
1649--003 Lisboa\\
Portugal
}
\email{%
ajm@fct.unl.pt
}

\thanks{For the second author, this work was developed within the project {\sc
  PEst-OE}/{\sc MAT}/{\sc UI}0143/2014 of CAUL, FCUL.}
\thanks{This work was partially supported by by the Funda\c{c}\~{a}o para
a Ci\^{e}ncia e a Tecnologia (Portuguese Foundation for Science and Technology) through the project {\sc UID}/{\sc
  MAT}/00297/2013 (Centro de Matem\'{a}tica e Aplica\c{c}\~{o}es). The authors thank the anonymous referee for their
very careful reading of the paper and for many helpful comments and suggestions.}

\begin{abstract}
  We give a combinatorial characterization of conjugacy in the sylvester monoid (the monoid of binary search trees),
  showing that conjugacy is decidable for this monoid. We then prove that conjugacy is undecidable in general for
  homogeneous monoids and even for multihomogeneous monoids.
\end{abstract}



\maketitle

\section{Introduction}

The notion of conjugation, which is extremely natural for groups, can
be generalized to semigroups in more than one way. Two elements $x$
and $y$ of a group $G$ are conjugate (in $G$) if there exists $g \in
G$ such that $y = g^{-1}xg$. A first attempt, due to
Lallement~\cite{lallement_semigroups}, to generalize this definition
to a semigroup $S$ is to define the \defterm{$\ell$-conjugacy} relation
$\sim_\ell$ by
\begin{equation}
\label{eq:lconjdef}
x \sim_\ell y \iff (\exists g \in S^1)(xg = gy).
\end{equation}
(where $S^1$ denotes the semigroup $S$ with an identity adjoined). It is easy to prove that $\sim_\ell$ is reflexive and
transitive, but not symmetric. A symmetric analogy, introduced by Otto~\cite{otto_conjugacy}, is \defterm{$o$-conjugacy}, defined
by
\begin{equation}
\label{eq:oconjdef}
x \sim_o y \iff (\exists g,h \in S^1)(xg = gy \land hx = yh).
\end{equation}
The relation $\sim_o$ is an equivalence relation. When the semigroup
$S$ contains a zero $0_S$, both $\sim_\ell$ and $\sim_o$ are the
universal relation $S \times S$, because taking $g = h = 0_S$ in
\eqref{eq:lconjdef} and \eqref{eq:oconjdef} shows that all elements of
$S$ are $\sim_\ell$- and $\sim_o$-related. This motivated Ara\'{u}jo,
Konieczny, and the second author~\cite{araujo_conjugation} to
introduce the $c$-conjugacy relation $\sim_c$, which is not simply the
universal relation when $S$ has a zero, but is genuinely useful. In
semigroups that do not contain zeroes, including the semigroups and
monoids we consider in this paper, $\sim_c$ and $\sim_o$ coincide, and
so we will not give the full definition of $\sim_c$.

%% defined by
%% \begin{align*}
%% x \sim_c y \iff{}& (\exists g \in \mathbb{P}^1(x), h \in \mathbb{P}^1(y))(xg = gy \land hx = yh),\\
%% &\text{where }\mathbb{P}(t) = \begin{cases}
%% \{g \in S : (\forall m \in S^1)(mag = 0 \implies ma = 0)\} \\
%% \qquad\qquad\text{if $S$ has a zero and $t \neq 0$,} \\
%% \mathrlap{\{0\}}\qquad\qquad\text{if $S$ has a zero and $t = 0$,}\\
%% \mathrlap{S}\qquad\qquad\text{if $S$ does not have a zero.}
%% \end{cases}
%% \end{align*}
%% When $S$ does not have a zero, $\sim_c$ and $\sim_o$ coincide. When
%% $S$ \emph{does} have a zero, $\sim_c$ is not the simply universal
%% relation.

Another approach is to define the \defterm{primary conjugacy} relation
$\sim_p$ by
\[
x \sim_p y \iff (\exists u,v \in S^1)(x = uv \land y = vu).
\]
However, $\sim_p$ is reflexive and symmetric, but not transitive; hence it is sensible to follow Kudryavtseva \&
Mazorchuk \cite{kudryavtseva_three,kudryavtseva_conjugation} in working with its transitive closure $\sim_p^*$. It is
easy to show that ${\sim_p} \subseteq {\sim_p^*} \subseteq {\sim_o} \subseteq {\sim_\ell}$. In some circumstances,
equality holds. For instance, in the free inverse monoid, ${\sim_p^*}$ and ${\sim_o}$ coincide when restricted to
non-idempotents \cite{choffrut_conjugacy}; this result has been extended to inverse monoids presented by a single
defining relation that has the form of a Dyck word \cite{silva_conjugacy}.

This paper is concerned with conjugacy in homogeneous monoids (that is, monoids with presentations where the two sides
of each defining relation have the same length) and multihomogeneous monoids (that is, monoids with presentations where
the two side of each defining relation contain the same number of each generator; see \fullref{\S}{sec:prelim} for the
formal definitions). Homogeneous monoids do not contain zeroes, so $\sim_c$ and $\sim_o$ coincide. The Plactic monoid
\cite[\ch 7]{lothaire_algebraic} and Chinese monoid \cite{cassaigne_chinese}, which are multihomogeneous monoids with
important connections to combinatorics and algebra, both have elegant combinatorial characterizations of conjugacy: in
both monoids, the relations ${\sim_p^*}$ and ${\sim_o}$ coincide, and two elements, expressed as words over the usual
set of generators, are conjugate (that is, ${\sim_p^*}$- and ${\sim_o}$-related) if and only if each generator appears
the same number of times in each word; see \cite[\S~4]{lascoux_plaxique} and \cite[Theorem~5.1]{cassaigne_chinese}.

Our first main result shows that the same characterization of
conjugacy holds for the sylvester monoid
(\fullref{Theorem}{thm:sylvestermonoidocongchar}). The sylvester monoid was
defined by Hivert, Novelli \& Thibon~\cite{hivert_algebra} as an
analogue of the plactic monoid where Schensted's algorithm for
insertion into Young tableaux (see \cite[ch.~5]{lothaire_algebraic})
is replaced by insertion into a binary search tree; from the sylvester
monoid, one then recovers the Hopf algebra of planar binary trees
defined by Loday \& Ronco \cite{loday_hopf}. Recently, the authors and
Gray proved that the standard presentations for the finite-rank
versions of sylvester monoids form (infinite) complete rewriting
systems, and that finite-rank sylvester monoids are biautomatic
\cite[\S~5]{cgm_chineseetc}.

We then prove that there exist homogeneous monoids where the problem
of deciding $o$-conjugacy is undecidable
(\fullref{Theorem}{thm:oconjundechomogeneous}). This strengthens a
result of Narendran \& Otto showing that $o$-conjugacy is undecidable
in general for monoids presented by finite complete rewriting systems
\cite[Lemma~3.6]{narendran_problems} and by almost-confluent rewriting
systems \cite[Theorem~3.4]{narendran_complexity}. (Homogeneous
presentations form almost-confluent rewriting systems
\cite[Proposition~3.2]{narendran_complexity}.) We then apply a
technique the authors developed with Gray \cite{cgm_homogeneous} to
deduce that there are multihomogeneous monoids in which the
$o$-conjugacy problem is undecidable
(\fullref{Theorem}{thm:oconjundecmultihomogeneous}).

\section{Preliminaries}
\label{sec:prelim}

\subsection{Homogeneous presentations and monoids}
\label{subsec:homogeneous}

Let $M$ be a monoid presented by $\pres{A}{\drel{R}}$. If $w,w' \in A^*$ represent the same element of $M$, we write $w
=_M w'$.

A monoid presentation $\pres{A}{\drel{R}}$ is \defterm{homogeneous} if
$|u| = |v|$ for all $(u,v) \in \drel{R}$. A monoid is
\defterm{homogeneous} if it can be defined by a homogeneous
presentation.

%% A presentation $\pres{A}{\drel{R}}$ is \defterm{$n$-homogeneous} if
%% $|u| = |v| = n$ for all $(u,v) \in \drel{R}$.

For a word $u \in A^*$ and a symbol $a \in A$, the number of symbols
$a$ in $u$ is denoted $|u|_a$. The \defterm{content} of $u \in A^*$ is
the function $a \mapsto |u|_a$. Two words $u,v \in A^*$ therefore have
the same content if $|u|_a = |v|_a$ for all $a \in A$. A presentation
$\pres{A}{\drel{R}}$ is \defterm{multihomogeneous} if $u$ and $v$ have
the same content for all $(u,v) \in \drel{R}$. A monoid is
\defterm{multihomogeneous} if it can be defined by a multihomogeneous
presentation. A multihomogeneous presentation is necessarily
homogeneous, and thus a multihomogeneous monoid is necessarily
homogeneous.

Let $M = \pres{A}{\drel{R}}$ be a homogeneous monoid and let $x \in
A^*$. Since the two sides of every defining relation in $\drel{R}$
have the same length, applying a defining relation does not alter the
length of a word. Thus if $y \in A^*$ and $y =_M x$, then $|y| =
|x|$.

Furthermore, if $A$ is finite, we can effectively compute the set $W_x$ of all words in $A^*$ equal to $x$ in $M$, since this
set only contains words that are the same length as $x$. In particular, this means that the word problem is soluble for
finitely generated homogeneous monoids: to decide whether $x$ and $y$ are equal in $M$, simply compute $W_x$ and check
whether $y$ is contained in $W_x$.

\subsection{Rewriting systems}

In this subsection, we recall the basic properties of string rewriting
systems needed for this paper. Fort further background reading, see
\cite{book_srs}.

A \defterm{string rewriting system}, or simply a \defterm{rewriting
  system}, is a pair $(A,\drel{R})$, where $A$ is a finite alphabet and
$\drel{R}$ is a set of pairs $(\ell,r)$, often written $\ell \imreduces
r$, known as \defterm{rewriting rules}, drawn from $A^* \times
A^*$. The single reduction relation $\imreduces$ is defined
as follows: $u \imreduces_{\drel{R}} v$ (where $u,v \in A^*$) if there
exists a rewriting rule $(\ell,r) \in \drel{R}$ and words $x,y \in A^*$
such that $u = x\ell y$ and $v = xry$. The reduction relation
$\reduces$ is the reflexive and transitive closure of
$\imreduces$. A word $w \in A^*$ is \defterm{reducible} if
it contains a subword $\ell$ that forms the left-hand side of a
rewriting rule in $\drel{R}$; it is otherwise called
\defterm{irreducible}.

The string rewriting system $(A,\drel{R})$ is \defterm{noetherian} if
there is no infinite sequence $u_1,u_2,\ldots \in A^*$ such that $u_i
\imreduces_{\drel{R}} u_{i+1}$ for all $i \in \nset$. The rewriting
system $(A,\drel{R})$ is \defterm{confluent} if, for any words $u,
u',u'' \in A^*$ with $u \reduces u'$ and $u
\reduces u''$, there exists a word $v \in A^*$ such that $u'
\reduces v$ and $u'' \reduces v$. A rewriting
system is complete if it is both confluent and Noetherian.

Let $(A,\drel{R})$ be a complete rewriting system. Then for any word $u
\in A^*$, there is a unique irreducible word $v \in A^*$ with $u
\reduces_{\drel{R}} v$ \cite[Theorem~1.1.12]{book_srs}. The irreducible
words are said to be in \defterm{normal form}. These irreducible words
form a cross-section of the monoid $\pres{A}{\drel{R}}$; thus this
monoid may be identified with the set of normal form words under the
operation of `concatenation plus reduction to normal form'.

\section{Sylvester monoids}

Let $A$ be the infinite ordered alphabet $\{1 < 2 < \ldots \}$. Let $\drel{R}$ be the (infinite) set of defining
relations
\[
\{(cavb,acvb) : a \leq b < c, v \in A^*\}.
\]
Then the \defterm{sylvester monoid}, denoted $S$, is presented by
$\pres{A}{\drel{R}}$~\cite[Definition~8]{hivert_algebra}. Note that
the presentation $\pres{A}{\drel{R}}$ is multihomogeneous.

A \defterm{(right strict) binary search tree} is a labelled rooted
binary tree where the label of each node is greater than or equal to
the label of every node in its left subtree, and strictly less than
every node in its right subtree; see the example in
\fullref{Figure}{fig:exbst}.

\begin{figure}[t]
\centering
\begin{tikzpicture}[every node/.style={circle,draw,inner sep=.7mm},level distance=10mm,level 1/.style={sibling distance=20mm},level 2/.style={sibling distance=10mm}]
  \node (root) {$4$}
    child { node (0) {$1$}
      child { node (00) {$1$} }
      child { node (01) {$3$}
        child { node (010) {$2$} }
        child { node (011) {$4$} } } }
    child { node (1) {$5$}
      child { node (10) {$5$} }
      child { node (11) {$6$} } };
\end{tikzpicture}
\caption{Example of a binary search tree $T$. The root has label $4$,
  so every label in the left subtree of the root is less than or equal
  to $4$ (and indeed the label $4$ does occur) and every label in the
  right subtree of the root is strictly greater than $4$. Notice that
  (for example) $T = \bst(265415314)$, and that $\lrp(T) =
  124315654$.}
\label{fig:exbst}
\end{figure}
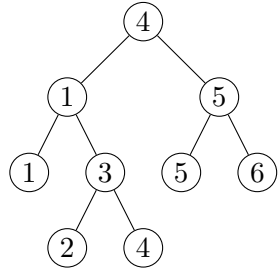

Given a binary search tree $\tree{T}$ and a symbol $a \in A$, one inserts $a$
into $\tree{T}$ as follows: if $\tree{T}$ is empty, create a node and label it
$a$. If $\tree{T}$ is non-empty, examine the label $x$ of the root node; if
$a \leq x$, recursively insert $a$ into the left subtree of the root
node; otherwise recursively insert $a$ into the right subtree of the
root node. Denote the resulting tree $a \cdot \tree{T}$. It is easy to see
that $a\cdot \tree{T}$ is also a binary search tree.

Given any word $w \in A^*$, define its corresponding binary search
tree $\bst(w)$ as follows: start with the empty tree and iteratively
insert the symbols in $w$ from right to left; again, see the example
in \fullref{Figure}{fig:exbst}.

The left-to-right postfix reading $\lrp(\tree{T})$ of a binary search
tree $\tree{T}$ is defined to be the word obtained as follows:
recursively perform the left-to-right postfix reading of the left
subtree of the root of $\tree{T}$, then recursively perform the
left-to-right postfix reading of the right subtree of the root of
$\tree{T}$, then output the label of the root of $\tree{T}$; again,
see the example in \fullref{Figure}{fig:exbst}. Note that
$\bst(\lrp(\tree{T})) = \tree{T}$ \cite[Proposition~15]{hivert_algebra}.

\begin{proposition}[{\cite[Theorem~10]{hivert_algebra}}]
\label{prop:bstcrossec}
Let $w,w' \in A^*$. Then $w =_{S} w'$ if and only if $\bst(w) =
\bst(w')$.
\end{proposition}

\begin{lemma}
\label{lem:ocongimpliessamecontent}
If two elements of $S$ are $o$-conjugate, they have the same
content.
\end{lemma}

\begin{proof}
Let $x,y \in A^*$ be such that $x \sim_o y$ in $S$. Then there exists
$g \in A^*$ such that $xg =_{S} gy$. Let $a \in A$. Since the
presentation $\pres{A}{\drel{R}}$ is multihomogeneous, $|xg|_a =
|gy|_a$. Thus $|x|_a + |g|_a = |g|_a + |y|_a$, and so $|x|_a =
|y|_a$. Since $a \in A$ was arbitrary, $x$ and $y$ have the same
content.
\end{proof}

Call a binary search tree \defterm{left full} if all of its nodes have empty right subtrees. That is, a binary search
tree is left full if every node is a left child of its parent node. (See \fullref{Figure}{fig:exbstfullleft}.) An
element $x \in S$ is \defterm{left full} if $\bst(x)$ is left full.

If $\tree{T}$ is a left full binary search tree, then by the definition of the left-to-right postfix reading,
\[
\lrp(\tree{T}) = 1^{k_1}2^{k_2}\cdots m^{k_m}
\]
for some $m \in \nset \cup \{0\}$ and $k_i \in \nset \cup \{0\}$. It is thus clear that there is exactly one left full
element with each content.

\begin{lemma}
\label{lem:samecontentimpliespcong}
If two elements of $S$ have the same content, they are
$\sim_p^*$-conjugate.
\end{lemma}

\begin{proof}
Since $\sim_p^*$ is an equivalence relation, it will
suffice to prove that every element of $S$ is
$\sim_p^*$-related to the left full element with the same content.

\begin{figure}[t]
\centering
\begin{tikzpicture}[every node/.style={circle,draw,inner sep=.7mm},level distance=10mm,level 1/.style={sibling distance=20mm},level 2/.style={sibling distance=10mm}]
  \node (root) at (0,0) {$5$};
  \node (0) at (-.5,-.5) {$5$};
  \node (00)  at (-1,-1) {$3$};
  \node (000) at (-1.5,-1.5) {$2$};
  \node (0000) at (-2,-2) {$2$};
  \node (00000) at (-2.5,-2.5) {$1$};
  \draw (root)--(0);
  \draw (0)--(00);
  \draw (00)--(000);
  \draw (000)--(0000);
  \draw (0000)--(00000);
\end{tikzpicture}
\caption{Example of a left full binary search tree, in this case $\bst(122355)$.}
\label{fig:exbstfullleft}
\end{figure}
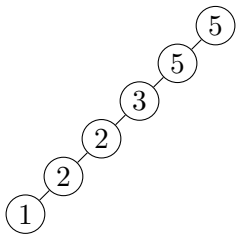

Consider a binary search tree $\tree{T}$. Starting from the root,
follow left child nodes until a node is found that has no left
child. Call this node $p_1$; this will be the minimal node of the
binary search tree. Let $p_2,\ldots,p_m$ be the successive ancestor
nodes of $p_1$ (that is, $p_{i+1}$ is the parent of $p_i$ for
$i=1,\ldots,m-1$, and $p_m$ is the root node). We define the
\defterm{left branch length} of $\tree{T}$ to be the maximal $k$ such
that none of $p_1,\ldots,p_k$ have right subtrees. Note that $p_1$ may
have a right subtree; in this case $\tree{T}$ has left branch length
$0$. Note further that $\tree{T}$ is left full if and only if its left
branch length is equal to the number of nodes in $\tree{T}$, or
equivalently equal to $|\lrp(\tree{T})|$. In general, the left branch
length of $\tree{T}$ is bounded above by $|\lrp(\tree{T})|$.

Let $h \in \nset$. We will prove the result for elements of $S$ of
length $h$ using reverse induction (from $h$ down to $0$) on the left
branch length of the corresponding binary search trees. Let $x \in S$
be such that $|x| = h$ and the left branch length of $\bst(x)$ is
$|x|$. Then, as noted above, $\bst(x)$ is a left full binary search
tree, and so $x$ is a left full element and there is nothing to
prove. This is the base case of the induction.

For the induction step, let $x \in S$ be such that $|x| = h$ and the
left branch length $k$ of $\bst(x)$ is strictly less than $|x|$, and
assume that every element of length $h$ whose binary search tree has
left branch length strictly greater than $k$ is $\sim_p^*$-related to
the left full element with the same content.

Let $\tree{T} = \bst(x)$, and let the $p_1,\ldots,p_k,\ldots,p_m$ be
as in the discussion above. Since $\tree{T}$ is not left full, $k$
(which is the left branch length of $\tree{T}$) is strictly less than
$h$, and $p_{k+1}$ has a right subtree, which we denote by
$\tree{T}_{k+1}$ for future reference. (See
\fullref{Figure}{fig:pickingy}.)  Note that $p_1 \leq p_2 \leq \ldots
\leq p_k$.

\begin{figure}[t]
\centering
\begin{tikzpicture}[x=11mm,y=11mm,every node/.style={circle,draw,inner sep=.3mm,minimum width=6mm},level distance=10mm,level 1/.style={sibling distance=20mm},level 2/.style={sibling distance=10mm}]
  \node (root) at (0,0) {};
  \node (0) at (-1,-1) {$p_{k+1}$};
  \node (00) at (-2,-2) {$p_2$};
  \node (000) at (-2.5,-2.5) {$p_1$};
  \node (01) at (0,-1.5) {};
  \node (010) at (-1,-2.5) {$q$};
  \node (0100) at (-2,-3.5) {$q$};
  \node (01000) at (-2.5,-4) {$q$};
  \node[forbidden sign] (01001) at (-1,-4.5) {$r$};
  \node[regular polygon,regular polygon sides=3] (0101) at (-.5,-3) {$w$};
  \draw[dashed] (root)--(0);
  \draw[dashed] (0)--(00);
  \draw (00)--(000);
  \draw (0)--(01);
  \draw[dashed] (01)--(010);
  \draw[dashed] (010)--(0100);
  \draw (010)--(0101);
  \draw (0100)--(01000);
  \draw[dashed] (0100)--(01001);
  \draw[dotted] (-3.5,-5)--(-3.5,-4.25)--(0,-0.75)--(0.5,-1.25)--(0.5,-5);
  \node[anchor=west,draw=none,rectangle] at (0.5,-1.25) {~subtree $\tree{T}_{k+1}$};
\end{tikzpicture}
\caption{The locations of $p_1,\ldots,p_{k+1}$, the nodes $q$, and the
  subtree with left-to-right postfix reading $u$. There is no vertex
  $r$ in the right subtree of any node $q$ below the uppermost.}
\label{fig:pickingy}
\end{figure}
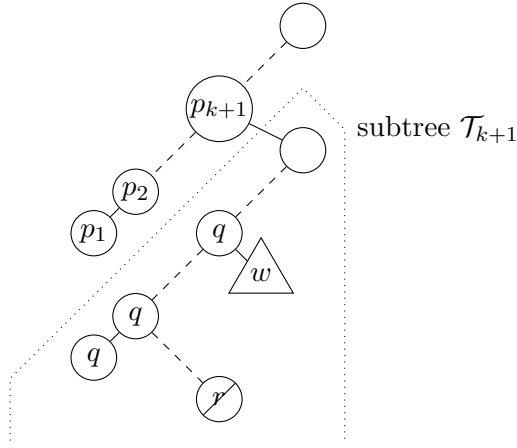

So $\lrp(\tree{T})$ is of the form $p_1p_2\cdots p_{k}up_{k+1}v$,
where $u$ is the left-to-right postfix reading of the non-empty right
subtree $\tree{T}_{k+1}$ of $p_{k+1}$, and $v$ is the left-to-right
postfix reading of the remainder of the tree above and to the right of
$p_{k+1}$ (and which will include $p_{k+2},\ldots,p_m$ in that order,
but not necessarily consecutively). Notice that if $m=k+1$, then
$p_{k+1}$ is the root note and $v = \emptyword$. On the other hand, if
$m>k+1$, then every symbol of the right subtree of $p_{k+1}$ is greater than $p_{k+1}$ and less than or equal to
$p_{k+2}$. Thus we deduce that $p_{k+1} < p_{k+2}$ and therefore
$p_{k+1}$ is less than every symbol in $v$. Therefore in both cases
$p_{k+1}$ is less than every symbol in $v$.

Starting from the root of $\tree{T}_{k+1}$, follow left child nodes until
a node is found that has no left child. Let $q$ be this node. Follow
the successive parent nodes until the uppermost node with value $q$ is
located, and let $w$ be the (possibly empty) left-to-right postfix
reading of the right subtree of this uppermost $q$. Suppose there are
$\ell$ such nodes $q$; note that $\ell \geq 1$. (See
\fullref{Figure}{fig:pickingy} again.)

Suppose for \textit{reductio ad absurdum} that one of the $\ell-1$
nodes $q$ below the uppermost has a right subtree. Let $r$ be some
symbol of this right subtree. Then $r > q$, since $r$ is in this right
subtree of this $q$, but $r \leq q$, since $r$ is in the left subtree
of the uppermost $q$; this is a contradiction. Hence only the
uppermost $q$ can have a non-empty right subtree.

Therefore $u = \lrp(\tree{T}_{k+1}) = q^{\ell - 1}wqu'$, where $u'$ is
the left-to-right postfix reading of the remainder of the
$\tree{T}_{k+1}$ above and to the right of the uppermost $q$.  Notice
that $p_{k+1} < q$, since $q$ is on the right subtree $\tree{T}_{k+1}$ of
$p_{k+1}$. Similarly, $q$ is less than every symbol of $w$. By the
choice of the uppermost node with value $q$ each symbol of $u'$ is
strictly greater than $q$.

As noted above, if $m = k+1$ then $v = \emptyword$ and so vacuously
$q$ is less than or equal to every symbol in $v$. If, on the other
hand, $m > k+1$, then $q$ is in the left subtree of $p_{k+2}$ and we
have $p_{k+1}<q\leq p_{k+2}$ and hence all the nodes above and to the
right of the node $p_{k+1}$ have value greater or equal than $q$: that
is, $q$ is less than or equal to every symbol in $v$.

Assume that $v$ contains $s$ instances of the symbol $q$. Let $\overline{v}$
denote the word obtained from $v$ by deleting the $s$ symbols $q$ (and
keeping the remaining symbols in the same order).

Thus we have
\begin{align*}
x ={}& p_1p_2\cdots p_{k} q^{\ell-1}wqu'p_{k+1}v \\
=_S{}& q^{\ell-1}wqu'p_1p_2\cdots p_{k} p_{k+1}v \\
&\qquad\text{[by multiple applications of $\drel{R}$ with $a$ being the $p_i$,} \\
&\qquad\quad\text{$b=p_{k+1}$, and $c$ being letters of $q^{\ell-1}wqu'$]} \displaybreak[0]\\
\sim_p{}& u'p_1p_2\cdots p_{k} p_{k+1}vq^{\ell-1}wq \displaybreak[0]\\
=_S{}& u'\overline{v}p_1p_2\cdots p_{k} p_{k+1}q^sq^{\ell-1}wq \\
&\qquad\text{[by multiple applications of $\drel{R}$ with $a$ being the $p_i$ or the $q$,}\\
&\qquad\quad\text{$b=q$, and $c$ being letters of $v$]} \displaybreak[0]\\
=_S{}& u'\overline{v}wp_1p_2\cdots p_{k} p_{k+1}q^sq^{\ell} \\
&\qquad\text{[by multiple applications of $\drel{R}$ with $a$ being the $p_i$ or the $q$,}\\
&\qquad\quad\text{$b=q$, and $c$ being letters of $w$]}
\end{align*}
[Note that we always have $a \leq b < c$ as required by the defintion
  of $\drel{R}$.]

When we apply the insertion algorithm to compute
$\bst(u'\overline{v}wp_1p_2\cdots p_{k} p_{k+1}q^sq^{\ell})$, the
$s+\ell$ symbols $q$ are inserted first, followed by the symbols
$p_{k+1},p_{k},\ldots,p_1$. So the root node is $q$, and, since $q >
p_{k+1} \geq p_{k} \geq \ldots \geq p_1 \geq p_1$, these symbols are
always inserted as left child nodes of at the leftmost node in the
tree.

\begin{figure}[t]
\centering
\begin{tikzpicture}[every node/.style={circle,draw,inner sep=.7mm},level distance=10mm,level 1/.style={sibling distance=20mm},level 2/.style={sibling distance=10mm}]
  \node (root) at (0,0) {$q$};
  \node (00)  at (-1,-1) {$q$};
  \node (000) at (-2,-2) {$p_{k+1}$};
  \node (00000) at (-3,-3) {$p_1$};
  \node[regular polygon,regular polygon sides=3] (01) at (1,-1) {$\phantom{w}$};
  \node[draw=none,anchor=west,rectangle,align=left] at (1.75,-1) {subtree containing symbols\\ from $u'\overline{v}w$};
  \draw[dashed] (root)--(00);
  \draw (00)--(000);
  \draw[dashed] (000)--(00000);
  \draw (root)--(01);
\end{tikzpicture}
\caption{Result of computing $\bst(u'\overline{v}wp_1p_2\cdots p_{k} p_{k+1}q^{\ell})$.}
\label{fig:computingbst}
\end{figure}

Since all of the symbols in $u'vw$ are strictly greater than $q$, they
are all inserted into the right subtree of the root node $q$. (See
\fullref{Figure}{fig:computingbst}.) So the left branch length of
$\bst(u'\overline{v}wp_1p_2\cdots p_{k} p_{k+1}q^sq^{\ell})$ is $(k+1)+(s+\ell-1)$,
which is greater than $k$ since $\ell\geq 1$. Hence, by the induction
hypothesis, $u'\overline{v}wp_1p_2\cdots p_{k} p_{k+1}q^sq^{\ell}$ is
$\sim_p^*$-related to the left full element with the same content as
$u'\overline{v}wp_1p_2\cdots p_{k} p_{k+1}q^sq^{\ell}$, which has the same content as
$x$. By transitivity, $x$ is $\sim_p^*$-related to the left full
element with the same content as $x$. Thus completes the induction
step and thus the proof.
\end{proof}

\begin{theorem}
\label{thm:sylvestermonoidocongchar}
In the sylvester monoid $\pres{A}{\drel{R}}$, two elements are
$o$-conjugate if and only if they have the same content (when viewed
as words). Moreover, $o$-conjugacy is decidable for the sylvester
monoid and ${\sim_o} = {\sim_p^*}$.
\end{theorem}

\begin{proof}
This is immediate from \fullref{Lemmata}{lem:ocongimpliessamecontent}
and \ref{lem:samecontentimpliespcong} and the fact that ${\sim_p^*}
\subseteq {\sim_o}$.
\end{proof}

\section{Homogeneous monoids}

\subsection{\texorpdfstring{Decidability of $\sim_p^*$-conjugacy}{Decidability of \textasciitilde\textunderscore p*-conjugacy}}

It is easy to see that $\sim_p^*$-conjugacy is decidable for finitely generated homogeneous monoids, as follows: Let $M
= \pres{A}{\drel{R}}$ be a homogeneous monoid with $A$ finite and let $x,y \in A^*$. We can compute the set $P_x$ of all
words $\sim_p^*$-conjugate to $x$ by setting $Y_0 = \{x\}$ and then iteratively computing $Y_{k}$ to be the set of words
of the form $vu$ for some $uv \in W_t$ for some $t \in Y_{k-1}$. Since $W_t$ is effectively computable (as discussed in
\fullref{Subsection}{subsec:homogeneous}), computing each $Y_k$ is effective. Since every set in the sequence $Y_0
\subseteq Y_1 \subseteq \ldots$ is contained in the finite set $A^{|x|}$, we must have $Y_k = Y_{k+1} = \ldots = P_x$
for some $k$. Hence computing $P_x$ is effective.

In particular, this means that the $\sim_p^*$-conjugacy problem is
soluble for homogeneous monoids: to decide whether $x$ and $y$ are
$\sim_p^*$-conjugate, simply compute $P_x$ and check whether $y$ is
contained in $P_x$.

\subsection{\texorpdfstring{Undecidability of $o$-conjugacy}{Undecidability of o-conjugacy}}

We are now going to prove that $o$-conjugacy is undecidable for
homogeneous monoids. We will extend this to multihomogeneous monoids
in the next subsection. This section assumes familiarity with Turing
machines and undecidability; see \cite[Chs~7--8]{hopcroft_automata}
for background.

\begin{theorem}
\label{thm:oconjundechomogeneous}
There exists a finitely presented homogeneous monoid $M =
\pres{A}{\drel{R}}$ in which the $o$-conjugacy problem is
undecidable. That is, there is no algorithm that takes as input
two words over $A$ and decides whether they are $o$-conjugate.
\end{theorem}

Let us outline the general strategy of the proof before beginning. The aim is to take a Turing machine with undecidable
halting problem, and build a finite complete rewriting system that simulates its computation. The rewriting system will
present a homogeneous monoid in such a way that the Turing machine halts on a given input $w$ if and only if the element
corresponding to the initial configuration of the machine with input $w$ is $o$-conjugate to the element corresponding
to a given halting configuration. However, the definition of the rewriting system is complicated by our need to build a
\emph{homogeneous} presentation.

\begin{proof}
Let $\tm{M} = (\Sigma,Q,\delta,q_0,q_a)$ be a deterministic Turing
machine with undecidable halting problem, where $\Sigma$ is the tape
alphabet (which contains a blank symbol $b$), $Q$ is the state set,
$\delta$ is the transition (partial) function, $q_0$ is the initial
state, and $q_a$ is the unique halting state.

For our purposes, a \defterm{configuration} of the Turing machine
$\tm{M}$ is a word $uqv$ with $u,v \in (\Sigma \cup \{b\})^*$ and $q
\in Q$. The symbol $q$ records the current state of the machine, the
word $uv$ records the contents of all cells of the tape that were
either initially non-blank or which have been visited by the
read/write head of the Turing machine, and the read/write head is
currently pointing at the cell corresponding to the leftmost symbol of
$v$. (Thus if the leftmost cell visited by read/write head is now
blank, then $u$ will begin with $b$. This is non-standard: normally, a
Turing machine configuration records only the smallest part of the
tape that includes all non-blank cells and the current position of the
read/write head. Thus one normally assumes that $u$ begins with a
symbol from $\Sigma$ and $v$ ends with a symbol from $\Sigma$.)

We will make some assumptions about $\tm{M}$; it is easy to see that
any deterministic Turing machine can be modified to give an equivalent
one satisfying these assumptions:
\begin{itemize}
\item $\tm{M}$ either changes the symbol on the tape where the
  read/write head is positioned, or moves the head left or
  right. (That is, it does not change the symbol \emph{and} move the
  head.) Thus we view the transition function as being a map from $Q
  \times \Sigma$ to $Q \times (\Sigma \cup \{L,R\})$.
\item $\tm{M}$ has no defined computation after entering the halting
  state $q_a$. That is, $\delta$ is undefined on $\{q_a\} \times
  \Sigma$.
\item $\tm{M}$ stores its initial tape contents in some part of the
  tape and, immediately before entering the halting state $q_a$,
  erases all symbols from the tape except this original content and
  moves to the left of the non-blank part of the tape. So $w \in
  L(\tm{M})$ if and only if $q_0w \tmcomp b^\alpha q_awb^\beta$ for
  some $\alpha,\beta$.
\end{itemize}

Let $\Sigma'$ be a set in bijection with $\Sigma$ under the map $x
\mapsto x'$. Let $A = \Sigma \cup \Sigma' \cup Q \cup \{h,h',d,z\}$; we are
going to define a rewriting system $(A,\drel{R})$. Let $\drel{R}$
consist of the following rewriting rules, for all $q_i \in Q$,
$s_i,s_j \in \Sigma$ and $x,y \in \Sigma \cup \{b\}$:

\smallskip
Group I (for $q_i,q_j \in Q$, $s_\ell \in \Sigma$ and $x,y \in \Sigma \cup \{b\}$):
\begin{align}
% 1
q_ixd &\to dq_js_\ell &\text{where }&(q_i,x)\delta = (q_j,s_\ell), \label{eq:1a}\\
q_ixdh &\to d^2q_jh &\text{where }&(q_i,x)\delta = (q_j,b), \label{eq:1b}\\
q_ixdy &\to dq_jb y &\text{where }&(q_i,x)\delta = (q_j,b), \label{eq:1c}\displaybreak[0]\\[1mm]
% 2
q_ihd &\to q_js_\ell h &\text{where }&(q_i,b)\delta = (q_j,s_\ell), \label{eq:2a}\\
q_ihd &\to dq_jh &\text{where }&(q_i,b)\delta = (q_j,b), \label{eq:2b}\displaybreak[0]\\[1mm]
% 3
q_ixd &\to dx'q_j &\text{where }&(q_i,x)\delta = (q_j,R), \label{eq:3}\displaybreak[0]\\
% 4
q_ihd &\to b'q_jh &\text{where }&(q_i,b)\delta = (q_j,R), \label{eq:4}\displaybreak[0]\\
% 5
y'q_ixd &\to dq_jyx &\text{where }&(q_i,x)\delta = (q_j,L), \label{eq:5}\displaybreak[0]\\[1mm]
% 6
s_\ell'q_ihd &\to dq_js_\ell h &\text{where }&(q_i,b)\delta = (q_j,L), \label{eq:6a}\\
b'q_ihd &\to d^2q_jh &\text{where }&(q_i,b)\delta = (q_j,L), \label{eq:6b}\displaybreak[0]\\[1mm]
% 7
h'q_ixd &\to h'q_jbx &\text{where }&(q_i,x)\delta = (q_j,L), \label{eq:7}\displaybreak[0]\\
% 8
h'q_ihd &\to dh'q_jh &\text{where }&(q_i,b)\delta = (q_j,L). \label{eq:8}
\end{align}

Group II (for $s_i \in \Sigma$):
\begin{align}
% 9
b'q_a &\to dq_a, \label{eq:9}\displaybreak[0]\\
% 10
q_as_id &\to dq_as_i. \label{eq:10}
\end{align}

Group III (for $s_i \in \Sigma$ and $x,y \in \Sigma \cup \{b\}$):
\begin{align}
% 11
xyd &\to xdy, \label{eq:11}\displaybreak[0]\\
% 12
s_ihd &\to s_idh, \label{eq:12}\displaybreak[0]\\
% 13
x'd &\to dx', \label{eq:13}\displaybreak[0]\\
% 14
h'd &\to dh'. \label{eq:14}
\end{align}

Group IV (for $s_i,s_j \in \Sigma$):
\begin{align}
% 15
s_ihz &\to s_izh, \label{eq:15}\displaybreak[0]\\
% 16
s_is_jz &\to s_izs_j, \label{eq:16}\displaybreak[0]\\
% 17
h'q_as_iz &\to zh'q_0s_i. \label{eq:17}
\end{align}
Let $M = \pres{A}{\drel{R}}$. Notice that $\pres{A}{\drel{R}}$ is a
homogeneous presentation, and so $M$ is homogeneous. We will now
proceed to show that the rewriting system $(A,\drel{R})$ is complete
(\fullref{Lemmata}{lem:noetherian} and \ref{lem:locallyconfluent}). We
will then show how the halting problem for $\tm{M}$ reduces to the
$o$-conjugacy problem for $M$.

\begin{lemma}
\label{lem:noetherian}
The rewriting system $(A,\drel{R})$ is noetherian.
\end{lemma}

\begin{proof}
Let $\leq$ be any partial order on the alphabet $A$ satisfying
\begin{align*}
&z < d < b < b' < s_i < h < q_k < s_j' < h' && \text{for $s_i,s_j \in \Sigma$ and $q_k \in Q$,} \\
&q_i = q_j &&\text{for $q_i,q_j \in Q$.}
\end{align*}
Let $\leq_\lenlex$ be the partial length-plus-lexicographic order induced by
$\leq$ on $A^*$.

By inspection, applying any rewriting rule in $\drel{R}$ results in a
strict reduction with respect to $\leq_\lenlex$. Since there are no
infinite descending chains in the partial length-plus-lexicographic
order, any process of rewriting using $\drel{R}$ must terminate. Hence
$(A,\drel{R})$ is noetherian.
\end{proof}

\begin{lemma}
\label{lem:locallyconfluent}
The rewriting system $(A,\drel{R})$ is locally confluent.
\end{lemma}

\begin{proof}
We will systematically examine possible overlaps of left-hand sides of
rules in $\drel{R}$ and show that they resolve.

\textit{Group~I--Group~I overlaps.} No non-trivial overlaps. This is
because symbols from $Q$ occur only once in these left-hand sides,
and exact overlaps of left-hand sides do not occur between rules of
different types because $\tm{M}$ is deterministic and so every
left-hand side determines a unique right-hand side.

\textit{Group~I--Group~II overlaps.} No non-trivial overlaps. This is
because symbols $q_a$ (the halting state of $\tm{M}$) do not occur on
the left-hand side of any Group~I rule, and symbols $b'$ and $d$ do
not occur at the start of left-hand sides of Group~I rules.

\textit{Group~I--Group~III overlaps.} Two possible overlaps, between
the left-hand sides of \eqref{eq:1c} and \eqref{eq:11}, and between
the left-hand sides of \eqref{eq:1c} and \eqref{eq:12}, both of which resolve:
\[
\begin{tikzpicture}[confluencediagram]
\node (start) at (0,0) {$q_ixdywd$};
\node (l1) at (-1,-1) {$dq_jbywd$};
\node (r1) at (1,-1) {$q_ixdydw$};
\node (finish) at (0,-2) {$dq_jbydw$};
\draw (start) edge[->] node[arrowlabel,left] {\eqref{eq:1c}} (l1);
\draw (start) edge[->] node[arrowlabel,right] {\eqref{eq:11}} (r1);
\draw (l1) edge[->] node[arrowlabel,left] {\eqref{eq:11}} (finish);
\draw (r1) edge[->] node[arrowlabel,right] {\eqref{eq:1c}} (finish);
\end{tikzpicture}
\qquad
\begin{tikzpicture}[confluencediagram]
\node (start) at (0,0) {$q_ixds_khd$};
\node (l1) at (-1,-1) {$dq_jbs_khd$};
\node (r1) at (1,-1) {$q_ixds_kdh$};
\node (finish) at (0,-2) {$dq_jbs_kdh$};
\draw (start) edge[->] node[arrowlabel,left] {\eqref{eq:1c}} (l1);
\draw (start) edge[->] node[arrowlabel,right] {\eqref{eq:12}} (r1);
\draw (l1) edge[->] node[arrowlabel,left] {\eqref{eq:12}} (finish);
\draw (r1) edge[->] node[arrowlabel,right] {\eqref{eq:1c}} (finish);
\end{tikzpicture}
\]
(Note that the symbols $x$ and $y$ in rules of type \eqref{eq:11} and
\eqref{eq:12} have been renamed here to $y$ and $w$, to avoid
conflicting with the $x$ in rules of type \eqref{eq:1c}. Similarly,
the subscript $i$ in rules of type \eqref{eq:12} has been renamed to
$k$ to avoid conflicting with the subscript $i$ in rules of type
\eqref{eq:1c}.) These are the only possible overlaps because the
symbols $d$ only appear once in each left-hand side, and in rules in
Group~III are preceded by two symbols in $\Sigma \cup \{b\}$, which
never happens in rules in Group~I. So overlaps cannot involve symbols
$d$. The only remaining possibilities are overlaps involving symbols
to the right of $d$ in Group~I rules, which happens precisely in the
cases we consider.

\textit{Group~I--Group~IV overlaps.} Two possible overlaps, between
the left-hand sides of \eqref{eq:1c} and \eqref{eq:15}, and between
the left-hand sides of \eqref{eq:1c} and \eqref{eq:16}, both of which resolve:
\[
\begin{tikzpicture}[confluencediagram]
\node (start) at (0,0) {$q_ixds_khz$};
\node (l1) at (-1,-1) {$dq_jbs_khz$};
\node (r1) at (1,-1) {$q_ixds_kzh$};
\node (finish) at (0,-2) {$dq_jbs_kzh$};
\draw (start) edge[->] node[arrowlabel,left] {\eqref{eq:1c}} (l1);
\draw (start) edge[->] node[arrowlabel,right] {\eqref{eq:15}} (r1);
\draw (l1) edge[->] node[arrowlabel,left] {\eqref{eq:15}} (finish);
\draw (r1) edge[->] node[arrowlabel,right] {\eqref{eq:1c}} (finish);
\end{tikzpicture}
\qquad
\begin{tikzpicture}[confluencediagram]
\node (start) at (0,0) {$q_ixds_ks_{k'}z$};
\node (l1) at (-1,-1) {$dq_jb s_ks_{k'}z$};
\node (r1) at (1,-1) {$q_izds_kzs_{k'}$};
\node (finish) at (0,-2) {$dq_jbs_kzs_{k'}$};
\draw (start) edge[->] node[arrowlabel,left] {\eqref{eq:1c}} (l1);
\draw (start) edge[->] node[arrowlabel,right] {\eqref{eq:16}} (r1);
\draw (l1) edge[->] node[arrowlabel,left] {\eqref{eq:16}} (finish);
\draw (r1) edge[->] node[arrowlabel,right] {\eqref{eq:1c}} (finish);
\end{tikzpicture}
\]
(Note that the subscript $j$ in rules of type \eqref{eq:16} has been
renamed here to $k'$, to avoid conflicting with the subscript $j$ in
rules of type \eqref{eq:1c}.)  These are the only possible overlaps
since symbols $z$ do not appear in rules in Group~I and symbols $d$ do
not occur in rules in Group~III. So overlaps cannot involve symbols
$z$. The only remaining possibilities are overlaps involving symbols
to the right of $d$ in Group~I rules, which happens precisely in the
cases we consider.

\textit{Group~II--Group~II overlaps.} By inspection, the only possible
overlaps are between the left-hand sides of rules of type \eqref{eq:9}
and type \eqref{eq:10}, and these overlaps resolve:
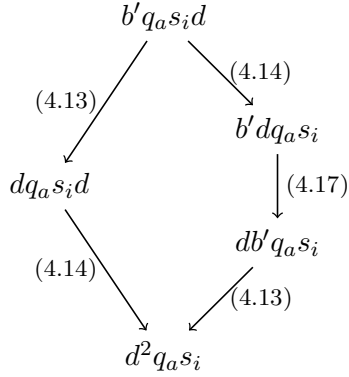
\[
\begin{tikzpicture}[confluencediagram]
\node (start) at (0,0) {$b'q_as_id$};
\node (l1) at (-1,-1.5) {$dq_as_id$};
\node (r1) at (1,-1) {$b'dq_as_i$};
\node (r2) at (1,-2) {$db'q_as_i$};
\node (finish) at (0,-3) {$d^2q_as_i$};
\draw (start) edge[->] node[arrowlabel,left] {\eqref{eq:9}} (l1);
\draw (start) edge[->] node[arrowlabel,right] {\eqref{eq:10}} (r1);
\draw (r1) edge[->] node[arrowlabel,right] {\eqref{eq:13}} (r2);
\draw (l1) edge[->] node[arrowlabel,left] {\eqref{eq:10}} (finish);
\draw (r2) edge[->] node[arrowlabel,right] {\eqref{eq:9}} (finish);
\end{tikzpicture}
\]

\textit{Group~II--Group~III overlaps.} There are no overlaps between
the left-hand side of a rule of type \eqref{eq:9} and a Group~III
rule, because $q_a$ does not appear in left-hand sides of Group~III
rules, and no left-hand side of a Group~III rule ends in $b'$. There
are no overlaps between left-hand side of a rule of type \eqref{eq:10}
and a Group~III rule since if $d$ is involved in the overlap, the
preceding letter of the Group~III rule must also be involved, which
limits us to rules of type \eqref{eq:11}; this is a contradiction
since in rules of type \eqref{eq:11} the symbol $d$ is preceded by two
symbols from $\Sigma \cup \{b\}$, not the symbols $q_a$ from the rule
of type \eqref{eq:9}.

\textit{Group~II--Group~IV overlaps.} There are no overlaps between
the left-hand sides of a Group~II rule and a Group~IV rule since the
left hand sides of Group~II rules start and end with symbols $b'$,
$q_a$, and $d$, and only $q_a$ appears in the left-hand side of a
Group~IV rule (in type \eqref{eq:17} rules), but between symbols $h'$
and $s_iz$, and $h'$ and $z$ do not appear on left-hand sides of
Group~II rules.

\textit{Group~III--Group~III overlaps.} There are no overlaps between
left-hand sides of Group~III rules, since they all end in symbols $d$
that appear nowhere else in the left-hand sides.

\textit{Group~III--Group~IV overlaps.} There are no overlaps between
left-hand sides of Group~III rules and Group~IV rules, since all
left-hand sides of Group~III end in symbols $d$, which do not appear
on left-hand sides of Group~IV rules, and left-hand sides of Group~IV
rules end in symbols $z$, which do not appear on left-hand sides of
Group~III rules.

\textit{Group~IV--Group~IV overlaps.} There are no overlaps between
left-hand sides of Group~IV rules, since they all end in symbols $z$
that appear nowhere else in the left-hand sides.

Since all overlaps resolve, the rewriting system $(A,\drel{R})$ is
locally confluent.
\end{proof}

By \fullref{Lemmata}{lem:noetherian} and \ref{lem:locallyconfluent}, the
rewriting system $(A,\drel{R})$ is complete.

\begin{lemma}
\label{lem:onestepcomputationandrewriting}
Let $u_1q_1v_1$ and $u_2q_2v_2$ be configurations of $\tm{M}$. Let
$\hat{v}_1$ and $\hat{v}_2$ be such that $v_1 = \hat{v}_1b^{\alpha_1}$
and $v_2 = \hat{v}_2b^{\alpha_2}$, where $\alpha_1$ and $\alpha_2$ are
maximal (possibly $0$). Then $u_1q_1v_1 \tmcomponestep u_2q_2v_2$ if
and only if
\[
h'u_1'q_1\hat{v}_1hd \reduces
d^{1+|u_1|-|u_2|+|\hat{v}_1|-|\hat{v}_2|}h'u_2'q_2\hat{v}_2h.
\]
Furthermore, in this case, $1+|u_1|-|u_2|+|\hat{v}_1|-|\hat{v}_2|$ is
either $0$, $1$, or $2$.
\end{lemma}

\begin{proof}
This is essentially a case-by-case checking using Group~I and
Group~III rules in $\drel{R}$. Essentially, the symbol $d$ moves to
the left of the $h$ using a rule of type
\eqref{eq:12} (note that by definition $\hat{v}_1$ does not end in a
symbol $b$ and thus ends with a symbol $s_i$ if it is non-empty) and then using rules
\eqref{eq:11} until it is one symbol to the right of $q_1$. Then a
rewrite using a Group~I rule occurs. (If $\hat{v}_1$ is the empty
word, so that $q_i$ and $h$ are adjacent, the symbol $d$ does not move
left: rules \eqref{eq:2a}, \eqref{eq:2b}, \eqref{eq:4}, \eqref{eq:6a},
or \eqref{eq:6b} apply immediately.) This produces $0$, $1$, or $2$
symbols $d$ which then move leftwards using rules \eqref{eq:13} and
\eqref{eq:14}. Note that
\begin{itemize}
\item Rules \eqref{eq:1a}, \eqref{eq:1c}, \eqref{eq:2b}, and
  \eqref{eq:8} correspond to cases where $|u_1| = |u_2|$ and
  $|\hat{v}_1| = |\hat{v}_2|$. The first three rules correspond to
  cases where the read/write head does not move and a blank does not
  replace a non-blank symbol at the extreme right of the non-blank
  portion of the tape. The last rule is simply a special case to cover
  the case when the entire tape is blank and the read/write head is
  moved left, which is not covered by over rules.
\item The rule \eqref{eq:3} corresponds to cases where $|u_1|+1 =
  |u_2|$ and $|\hat{v}_1|-1 = |\hat{v}_2|$. This is the case where the
  read/write head moves to the right somewhere in the non-blank part
  of the tape.
\item Rules \eqref{eq:5} and \eqref{eq:6a} correspond to cases where
  $|u_1|-1 = |u_2|$ and $|\hat{v}_1|+1 = |\hat{v}_2|$. This is the
  case where the read/write head moves to the left somewhere in the
  non-blank part of the tape.
\item The rule \eqref{eq:4} corresponds to cases where $|u_1|+1 =
  |u_2|$ and $|\hat{v}_1| = |\hat{v}_2|$. This is the case where the
  read/write head is pointing to a blank symbol somewhere to the right
  of the non-blank part of the tape, and then moves further right.
\item The rule \eqref{eq:6b} corresponds to cases where $|u_1|-1 =
  |u_2|$ and $|\hat{v}_1| = |\hat{v}_2|$. This is the case where the
  read/write head is pointing to a blank symbol at least one symbol to
  the right of the non-blank part of the tape, and then moves left.
\item Rules \eqref{eq:2a} and \eqref{eq:7} corresponds to cases where
  $|u_1| = |u_2|$ and $|\hat{v}_1|+1 = |\hat{v}_2|$. The first rule is
  where the read/write head is pointing either to the first blank
  symbol to the right of the non-blank part of the tape, or to some
  blank symbol further right, and replaces it with a non-blank
  symbol. The second rule is where the read/write head is pointing
  either to the leftmost non-blank symbol or some blank symbol,
  then moves further left.
\item The rule \eqref{eq:1b} corresponds to cases where $|u_1| =
  |u_2|$ and $|\hat{v}_1|-1 = |\hat{v}_2|$. This rule is where the
  read/write head is pointing to the rightmost non-blank symbol on the
  tape, and replaces it with a blank.
\end{itemize}
\end{proof}

\begin{lemma}
\label{lem:computationandrewriting}
Using notation from
\fullref{Lemma}{lem:onestepcomputationandrewriting}, $u_1q_1v_1
\tmcomp u_2q_2v_2$ if and only if there is some natural $\gamma$ such that
\[
h'u_1'q_1\hat{v}_1hd^\gamma \reduces
d^{\gamma+|u_1|-|u_2|+|\hat{v}_1|-|\hat{v}_2|}h'u_2'q_2\hat{v}_2h.
\]
In particular, $q_0w \tmcomp (b')^\alpha q_a wb^\beta$ if and only if
there is some natural $\gamma$ such that
\begin{equation}
\label{eq:computationandrewriting}
h'q_0whd^\gamma \reduces
d^{\gamma-\alpha}h'(b')^\alpha q_awh \reduces d^\gamma h'q_awh.
\end{equation}
\end{lemma}

\begin{proof}
The first part follows by induction using
\fullref{Lemma}{lem:onestepcomputationandrewriting}. In
\eqref{eq:computationandrewriting}, the first reduction is simply a
particular case of the first part. The second reduction uses rules
\eqref{eq:9}, \eqref{eq:13}, and \eqref{eq:14}.
\end{proof}

We have now established the correspondence between rewriting using
$(A,\drel{R})$ and computation in $\tm{M}$. We can now prove the
connection with $o$-conjugacy:

\begin{lemma}
\label{lem:inltmiffoconj}
For $w \in \Sigma^+$, we have $w \in L(\tm{M})$ if and only $h'q_0wh
\sim_o h'q_awh$.
\end{lemma}

\begin{proof}
Suppose $w \in L(\tm{M})$. Then $q_0w \tmcomp q_aw$. So by
\fullref{Lemma}{lem:computationandrewriting}, $h'q_0whd^\alpha
\reduces d^\alpha h'q_awh$ for some natural $\alpha$. Furthermore,
using Group~IV rules, $h'q_awhz \reduces zh'q_0wh$. Hence $h'q_0wh
\sim_o h'q_awh$.

Suppose now that $h'q_0wh \sim_o h'q_awh$. Then there exists a word $u
\in \Sigma^*$ such that $h'q_0whu \thue uh'q_awh$. Assume without loss
that $u$ is irreducible. First note that $uh'q_awh$ is irreducible,
since no rule in $\drel{R}$ can be applied to $h'q_awh$ since no
symbols $d$, $z$, or $b'$ are present, and there is no rule whose
left-hand side contains $h'$ except at the start. Hence $h'q_0whu
\reduces uh'q_awh$.

Suppose first $u$ contains a symbol $z$. Then $u$ factors as $u_1zu_2$
where $u_1$ does not contain $z$. The word $h'q_0wh$ is irreducible
since no symbols symbols $d$, $z$, or $b'$ are present, so rewriting
$h'q_0whu$ must begin with a rewriting rule that includes the
distinguished symbol $h$ and a non-empty prefix of $u$. In rules in
$\drel{R}$, the symbol $h$ is only followed by a symbol $d$ or
$z$. Since $u_1$ does not include $z$, the first symbol of $u_1$ must
therefore be $d$. As in
\fullref{Lemma}{lem:onestepcomputationandrewriting}, this $d$ will
move to the left, where it will either disappear, or one or two
symbols $d$ will emerge to the left of the $h'$. The $h$ remains next
to the remainder of $u_1$, so by the same reasoning the next symbol of
$u_1$ must also be $d$. Repeating this reasoning, we see $u_1 =
d^\alpha$ and $h'q_0whu = h'q_0whd^\alpha zu_2 \reduces d^\beta
h'\cdots q_c \cdots hzu_2$ for some non-negative integers $\alpha$ and
$\beta$. Notice that $d^\beta h'\cdots q_c \cdots h$ is irreducible,
so further rewriting must use a rule of type \eqref{eq:15} followed by
rules of type \eqref{eq:16} to move the $z$ to the
left until it is one symbol to the right of $q_c$. To apply the rule
\eqref{eq:17}, which is necessary if we want to obtain $uh'q_0wh$,
where all the symbols $z$ are to the left of the symbol from $Q$, is
only possible if $q_c = q_a$. By
\fullref{Lemma}{lem:computationandrewriting}, this implies that
reading $w$ causes $\tm{M}$ to enter the halting state $q_a$; and
hence $w$ is accepted by $L(\tm{M})$.

If $u$ does not contain a symbol $z$, then by the same reasoning as in
the last paragraph, $u=d^\alpha$ and so $hq_0wh'd^\alpha \reduces
d^\alpha hq_awh'$. So $q_0w \tmcomp q_aw$ by
\fullref{Lemma}{lem:computationandrewriting} and so $w \in L(\tm{M})$.
\end{proof}

By \fullref{Lemma}{lem:inltmiffoconj}, the problem of whether $\tm{M}$
halts on a given input reduces to the $o$-conjugacy problem for
$M$. Hence the $o$-conjugacy problem for $M$ is undecidable.
\end{proof}

Since $\sim_p^*$ is decidable for homogeneous monoids and $\sim_o$ is
undecidable in general for homogeneous monoids by
\fullref{Theorem}{thm:oconjundechomogeneous}, we have (very
indirectly) proved the following corollary:

\begin{corollary}
In the class of homogeneous monoids, $\sim_p^*$ and $\sim_o$ do not
coincide in general.
\end{corollary}

\subsection{Multihomogeneous monoids}

\begin{theorem}
\label{thm:oconjundecmultihomogeneous}
There exists a finitely presented multihomogeneous monoid $N =
\pres{X}{\drel{S}}$ in which the $o$-conjugacy problem is
undecidable. That is, there is no algorithm that takes as input
two words over $X$ and decides whether they are $o$-related.
\end{theorem}

\begin{proof}
Let $M = \pres{A}{\drel{R}}$ be a finitely presented homogeneous
monoid with undecidable $o$-conjugacy problem; such a monoid
exists by \fullref{Theorem}{thm:oconjundechomogeneous}. Suppose $A =
\{a_1,\ldots,a_n\}$. Let $X = \{x,y\}$. Define a map
\[
\phi : A^* \to X^*;\qquad a_i \mapsto x^2y^ixy^{n-i+1}.
\]
Let
\[
\drel{S} = \drel{R}\phi = \bigl\{(u\phi,v\phi) : (u,v) \in \drel{R}\bigr\}.
\]
Let $N = \pres{X}{\drel{S}}$. Then $\pres{X}{\drel{S}}$ is a
multihomogeneous presentation, and thus $N$ is a multihomogeneous
monoid: the proof is not difficult; see \cite[Proposition~5.8]{cgm_homogeneous} for
details. Furthermore, it is easy to prove that $\phi$ embeds $M$ into $N$
\cite[Proposition~5.9]{cgm_homogeneous}.

Suppose $p,q \in M$ and $w \in N$ are such that $(p\phi)w =_N
w(q\phi)$. Since $(p\phi)w$ contains a prefix of length $|p\phi|$ that
is in $\im\phi$, a prefix of $w(q\phi)$ of length $|p\phi|$ is in
$\im\phi$ since application of relations in $\drel{S}$ preserves
subwords that lie in $\im\phi$; see \cite[Proof of Proposition~5.9]{cgm_homogeneous} for further
details. Hence $(p\phi)w$ contains a prefix of length $2|p\phi|$ in
$\im\phi$, and so $w(q\phi)$ contains a prefix of length $2|p\phi|$ in
$\im\phi$ by the same lemma. Iterating this process, eventually we see
that $w(q\phi)$ contains a prefix of length at least $|w|+2$ that lies
in $\im\phi$. So suppose $v = v_1v_2\cdots v_m \in A^*$ (where $v_i
\in A$) is such that $v\phi$ is a prefix of $w(q\phi)$ of length at
least $|w|+2$. The subword $q\phi$ begins with $x^2$, and so
$w(q\phi)$ beings with $wx^2$. Since $|v\phi| \geq |w|+2$, this $x^2$
must be in the prefix $v\phi$ of $w(q\phi)$. Subwords $x^2$ only occur
in $v\phi$ at the start of the subwords $v_i\phi$, and so $w =
(v_1\cdots v_{m'})\phi$ for some $m' < m$. Thus $(pv_1\cdots
v_{m'})\phi =_N (v_1\cdots v_{m'}q)\phi$. Since $\phi$ is an
embedding, $pv_1\cdots v_{m'} =_M v_1\cdots v_{m'}q$.

Therefore, $o$-conjugacy on $\im\phi$ is simply $o$-conjugacy in
$N$ restricted to $\im\phi$. Since $o$-conjugacy is undecidable for
$M$ and thus for $\im\phi$, it is therefore undecidable for the
multihomogeneous monoid $N$.
\end{proof}

% Bibliography

\bibliography{automaticsemigroups,languages,presentations,rewriting,semigroups,\jobname,c_publications}
%\bibliography{\jobname_ext}
\bibliographystyle{alphaabbrv}

\end{document}